\documentclass[letterpaper, 10 pt, conference]{ieeeconf}  

\IEEEoverridecommandlockouts                              

\usepackage{amsmath} 
\usepackage{amssymb}  
\usepackage[colorlinks=true,linkcolor=blue]{hyperref}
\usepackage{amsfonts}
\usepackage{mathrsfs}
\usepackage{xcolor}
\usepackage{color}
\usepackage{graphicx}
\usepackage{algorithm}
\usepackage{algorithmic}
\newcommand{\norm}[2][]{\|#2\|_{#1}}

\newcommand \RR {\mathbb{R}}

\newcommand \EE {\mathbb{E}}

\DeclareMathOperator{\sign}{sign}
\newcommand{\scp}[2]{\langle #1 \,,\, #2\rangle}

\newtheorem{theorem}{Theorem}[section]

\newtheorem{lemma}[theorem]{Lemma}
\newtheorem{definition}[theorem]{Definition}
\newtheorem{remark}[theorem]{Remark}
\newtheorem{example}{Example}[section]

\newcommand{\eqdef}{\overset{\text{def}}{=}}

\graphicspath{ {./figures/} }
\usepackage{subfig}

\title{\LARGE \bf
An accelerated randomized Bregman-Kaczmarz method for strongly convex linearly constraint optimization*
}

\author{Lionel Tondji$^{1}$,  Dirk A. Lorenz$^{1}$ and Ion Necoara$^{2,3}$
\thanks{$^{1}$Institute of Analysis and Algebra, TU Braunschweig, 38092 Braunschweig, Germany, \texttt{\small l.ngoupeyou-tondji@tu-braunschweig.de}, \texttt{\small d.lorenz@tu-braunschweig.de}.}%
\thanks{$^{2}$Automatic Control and Systems Engineering Department, University Politehnica Bucharest,
060042 Bucharest, Romania, {\tt\small ion.necoara@upb.ro}}%
\thanks{$^{3}$Gheorghe Mihoc-Caius Iacob Institute of Mathematical Statistics and Applied Mathematics of the Romanian Academy, 050711 Bucharest, Romania.} 
}

\begin{document}

\maketitle
\thispagestyle{empty}
\pagestyle{empty}

\begin{abstract}
  In this paper, we propose a randomized accelerated method for the minimization of a strongly convex function under linear constraints. The method is of Kaczmarz-type, i.e. it only uses a single linear equation in each iteration. To obtain acceleration we build on the fact that the Kaczmarz method is dual to a coordinate descent method. We use a recently proposed acceleration method for the randomized coordinate descent and transfer it to the primal space.  This method inherits many of the attractive features of the accelerated coordinate descent method, including its worst-case convergence rates. A theoretical analysis of the convergence of the proposed method is given. Numerical experiments show that the proposed method is more efficient and faster than the existing methods for solving the same problem.

\end{abstract}

\section{Introduction}
\noindent We consider the fundamental problem of approximating solutions of large scale linear systems of the form:
\begin{equation}
\label{eq:LS}
\mathbf{A}x = b
\end{equation}
with matrix $\mathbf{A} \in \RR^{m \times n}$ and right hand side $b \in \RR^m$. We consider the case that the full matrix is not accessible simultaneously, but that one can only work with single rows of the system~\eqref{eq:LS} at a time.
Such problems arise in several fields of engineering and physics problems, such as sensor networks,
signal processing, partial differential equations, filtering, computerized tomography, optimal control, inverse problems and machine learning, to name just a few~\cite{tropp2011improved,olshanskii2014iterative,rabelo2022kaczmarzillposed,hounsfield1973computerized,jiao2017preasymptotic,patrascu2017nonasymptotic}.
Given the possibility of multiple solutions of~\eqref{eq:LS}, we set out to find the unique solution, characterized by the function $f$, i.e.
\begin{equation}
\label{eq:PB}
f^* \eqdef \min_{x \in \RR^n} f(x) \quad \text{subject to} \quad  \mathbf{A}x=b,
\end{equation}
with a strongly convex function $f$. However, we will not assume smoothness of $f$. One possible example is $f(x) = \lambda \cdot \|x\|_1 + \frac{1}{2}\|x\|^2_2$ and it is known that this function favors sparse solutions for appropriate choices of $\lambda > 0$, see~\cite{cai2009linearized,yin2010analysis,lorenz2014linearized,LWSM14,SL19}.  We assume throughout the paper that both $m$ and $n$ are large and that the system is consistent. We denote by $a_{i}^{T}$ the rows of $\mathbf{A}$ and assume $a_i \neq 0$ for all $i \in [m]:=\{1,\dots,m\}$. Since $f$ is strongly convex, problem~\eqref{eq:PB} has a unique solution $\hat x$. In applications,  it is usually sufficient to find a point which is not too far from $\hat x$. In particular, one chooses the error tolerance $\varepsilon >0$ and aims to find a point $x$ satisfying $\|x - \hat x\|_{2} \leq \varepsilon$. Since our method will be a stochastic method, and hence, the iterates $x$ are random vectors. Hence, our goal will be to compute approximate solutions of~\eqref{eq:PB} that fulfills $\mathbb{E}[\|x - \hat x\|_{2}] \leq \varepsilon$, where $\mathbb{E}[ \cdot]$ denotes the expectation w.r.t. randomness of the algorithm.
\vspace{-0.35cm}

\subsection{Related work}
\noindent The linear system~\eqref{eq:LS} may be so large that full matrix operations are very expensive or even infeasible. Then, it appears desirable to use iterative algorithms
with low computational cost and storage per iteration that produce good approximate solutions of~\eqref{eq:PB} after relatively few iterations.  The Kaczmarz method~\cite{Kac37} and its randomized variants~\cite{SV09,gower2015randomized,gower2019adaptive,necoara2019faster} are used to compute the minimum $\ell_2$-norm solutions of consistent linear systems.  In each iteration $k$, a row vector $a_i^\top$
of $\mathbf{A}$ is chosen at random from the system~\eqref{eq:LS} and the current iterate $x_k$ is projected onto the solution space of that equation to obtain $x_{k+1}$. Note that this update rule requires low cost per iteration and storage of order $\mathcal{O}(n)$. Recently, a new variant of the randomized Kaczmarz (RK) method namely the randomized sparse Kaczmarz method (RSK)~\cite{LWSM14,SL19} with almost the same low cost and storage requirements has shown good performance in approximating sparse solutions of large consistent linear systems. The papers~\cite{lorenz2014linearized,SL19} analyze this method by interpreting it as a sequential, randomized Bregman projection method (where the Bregman projection is done with respect to the function $f$) while~\cite{P15} connects it with the coordinate descent method via duality.
Variations of RSK including block/averaging variants~\cite{needell2014paved,P15,necoara2019faster,du2020randomized,miao2022greedy}, averaging methods~\cite{tondji2022faster} and adaptions to least squares problems are given in \cite{zouzias2013randomized,Du19,schopfer2022extended}. In those variants, one usually needs to have access to more than one row of the matrix $\mathbf{A}$ at the cost of increasing the  memory. Coordinate descent methods are often considered in the context of minimizing composite convex objective functions \cite{necoara2016,nesterov2012efficiency,fercoq2015accelerated}. At each iteration, they update only one coordinate of the vector of variable, hence using partial derivatives rather than the whole
gradient. A randomized accelerated coordinate descent method  was proposed in~\cite{nesterov2012efficiency} for smooth functions. Accelerated methods transform the proximal coordinate descent, for which the optimality gap decreases as $\mathcal{O}(1/k)$, into an
algorithm with optimal $\mathcal{O}(1/k^2)$ complexity under mild additional computational cost~\cite{nesterov2012efficiency,fercoq2015accelerated}.


\subsection{Contribution}
\label{sec:contributions}
\noindent To the best of our knowledge, accelerated Bregman-Kaczmarz variants have not yet been
proposed and analyzed in the literature. Specifically,
we make the following contributions:
\begin{itemize}
    \item Beyond interpreting the Bregman-Kaczmarz method as a dual coordinate descent, we propose an accelerated variant using only one row of the matrix. We refer to it as Accelerated Randomized Bregman-Kaczmarz method (ARBK) with more details in Algorithm~\ref{alg:ARBK}.

    \item  By exploiting the connection between primal and dual updates, we obtain convergence as a byproduct and convergence rates which have not been available so far. We prove that our accelerated method  leads to faster convergence than its standard counterpart. 

    \item We also validate this empirically and we provide implementations of our algorithm in \texttt{Python}.
\end{itemize}

\subsection{Outline.}
\noindent The remainder of the paper is organized as follows. Section~\ref{sec:basicnotions} provides notations, a brief overview on convexity and Bregman distances. In Section~\ref{sec:interpretation} we state our method and we give an interpretation of it in the dual space. Section~\ref{sec:convergence} provides convergence guarantees for our proposed method. In Section~\ref{sec:numerics}, numerical experiments demonstrate the effectiveness of our method and provide insight regarding its behavior and its hyper-parameters. Finally,
Section~\ref{sec:conclusion} draws some conclusions.


\section{Notation and basic notions}
\label{sec:basicnotions}
\noindent For integers $m$ we denote $[m] \eqdef \{1,2,\dots,m\}$. Given a symmetric positive definite matrix $\mathbf{B},$ we denote the induced inner product by 
\begin{equation*}
   \langle x, y \rangle_{\mathbf{B}} \eqdef  \langle x, \mathbf{B} y \rangle = \sum_{i,j \in [n]} x_i \mathbf{B}_{ij} y_j , \quad x, y \in \RR^{n}
\end{equation*}
and its induced norm by $\|.\|_{\mathbf{B}}^2 \eqdef \langle \cdot, \cdot \rangle_{\mathbf{B}}$ and use the short-hand notation $\|.\|_2$ to mean $\|.\|_{\mathbf{I}}$. For an $n\times m$ real matrix $\mathbf{A}$ we denote by $\mathcal{R}(\mathbf{A}), \|\mathbf{A} \|_F$ and $a_i^\top$ its range space, its Frobenius norm and its $i$-th row, respectively. By $e_i$ we denote the $i$th column of the identity matrix $\mathbf{I}_n \in \RR^{n \times n}$. For a random vector $x_i$ that depends on a random index $i$ (where $i$ is chosen with probability $p_{i}$) we denote $\EE [x_i] \eqdef \sum_{i\in [q]} p_ix_i$ and we will just write $\EE [x_i]$ when the probability distribution is clear from the context. Given a vector $x \in \RR^n$, we define the soft shrinkage operator, which acts componentwise on a vector $x$ as
\begin{equation} \label{eq:S}
\big(S_{\lambda}(x))_j = \max\{|x_j|-\lambda,0\} \cdot \sign(x_j)\,.
\end{equation}

Now we collect some basic notions on convexity and the Bregman distance.
Let $f:\RR^n \to \RR$ be convex (note that we assume that $f$ is finite everywhere, hence it is also continuous). The \emph{subdifferential} of $f$ at any $x \in \RR^n$ is defined by
\[
\partial f(x) \eqdef \{x^* \in \RR^n| f(y) \ge f(x) + \langle x^*, y-x \rangle, \forall\, y \in \RR^n \},
\]
which is nonempty, compact and convex. The function  $f:\RR^n \to \RR$ is said to be \emph{$\alpha$-strongly convex}, if for all $x,y \in \RR^n$ and subgradients $x^* \in \partial f(x)$ we have
\[
f(y) \ge f(x) + \scp{x^*}{y-x} + \tfrac{\alpha}{2} \cdot \norm[2]{y-x}^2 \,.
\]
If $f$ is $\alpha$-strongly convex, then $f$ is coercive, i.e.
\[
\lim_{\norm[2]{x} \to \infty} f(x)=\infty \,,
\]
and its \emph{Fenchel conjugate} $f^{*}:\RR^n \to \RR$ given by 
\[
f^*(x^*)\eqdef \sup_{y \in \RR^n} \scp{x^*}{y} - f(y)
\]
is also convex, finite everywhere and coercive. Additionally, $f^*$ is differentiable with a \emph{Lipschitz-continuous gradient} with constant $L_{f^*}=\frac{1}{\alpha}$, i.e. for all $x^*,y^* \in \RR^n$ we have
\[
\norm[2]{\nabla f^*(x^*)-\nabla f^*(y^*)} \le L_{f^*} \cdot \norm[2]{x^*-y^*} \,,
\]
which implies the estimate
\begin{align} \small \label{eq:Lip}
f^*(y^*)\!
\le \!f^*(x^*)\! -\!\scp{\nabla f^*(x^*)}{y^*-x^*}\! +\! \tfrac{L_{f^*}}{2}\norm[2]{x^*\!\!-y^*}^2. 
\end{align}
The function $f^*$ is said to have componentwise Lipschitz continuous gradient if
\[
|\nabla_i f^*(x^* + he_i) - \nabla_i f^*(x^*)| \leq L_{f^*,i} \cdot |h|,
\]
for all $x^* \in \RR^n, h \in \RR, i\in [n].$ 

\begin{definition} \label{def:D}
The \emph{Bregman distance} $D_f^{x^*}(x,y)$ between $x,y \in \RR^n$ with respect to $f$ and a subgradient $x^* \in \partial f(x)$ is defined as
\[
D_f^{x^*}(x,y) \eqdef f(y)-f(x) -\scp{x^*}{y - x}\,.
\]
\end{definition}

Fenchel's equality states that $f(x) + f^*(x^*) = \scp{x}{x^*}$ if $x^*\in\partial f(x)$ and implies that the Bregman distance can be written as
\[
D_f^{x^*}(x,y) = f^*(x^*)-\scp{x^*}{y} + f(y)\,.
\]

\begin{example}\cite{SL19}
\label{exmp:f}
The objective function
\begin{equation} \label{eq:spf}
f(x) \eqdef \lambda \cdot \norm[1]{x} + \tfrac{1}{2} \cdot \norm[2]{x}^{2}
\end{equation}
is strongly convex with constant $\alpha=1$ and its conjugate function can be computed with the soft shrinkage operator from Eq~\eqref{eq:S}
\[ 
f^{*}(x^{*}) = \tfrac{1}{2} \cdot \norm[2]{S_{\lambda}(x^{*})}^{2}, \quad \mbox{with} \quad \nabla f^{*}(x^{*}) = S_{\lambda}(x^{*}) \,.
\]
It Fenchel conjugate $f^*$ has componentwise Lipschitz gradient with constants $L_{f^*,i} = 1$ and for any $x^*=x+\lambda \cdot s \in \partial f(x)$ we have
\[
D_f^{x^*}(x,y)=\frac{1}{2} \cdot \norm[2]{x-y}^2 + \lambda \cdot(\norm[1]{y}-\scp{s}{y}) \,.
\]
which give us $D_f^{x^*}(x,y)=\frac{1}{2}\norm[2]{x-y}^2$ for $\lambda = 0.$
\end{example}

\medskip 

\noindent The following inequalities are crucial for the convergence analysis of the randomized algorithms.
They immediately follow from the definition of the Bregman distance and the assumption of strong convexity of $f$, see~\cite{LSW14}.
For  $x,y \in \RR^n$ and $x^* \in \partial f(x)$, $y^* \in \partial f(y)$ we have
\begin{align} \label{eq:D}
&\frac{\alpha}{2} \norm[2]{x-y}^2  
\le  D_f^{x^*}(x,y) \le \scp{x^*-y^*}{x-y}, 
\end{align}

\section{Coordinate descent and  Bregman-Kaczmarz method}
\label{sec:interpretation}
\noindent Note that by a proper scalling of $f$ we can always assume $\alpha = 1$. Therefore, in the
sequel we consider 1-strongly convex function $f$. In particular, using the conjugate of $f$ we can write the dual function for the problem \eqref{eq:PB} as:
\begin{align*}
 \Psi(y) &= \inf_x \mathcal{L}(x,y) \\
 &= \inf_x (f(x) - y^\top(\mathbf{A}x-b) ) \\
 &= b^\top y - f^*(\mathbf{A}^\top y),
\end{align*}
where $\mathcal{L}(x,y)$ denotes the Lagrangian function. The dual problem of \eqref{eq:PB} is given by:
\begin{equation}
\label{eq:DP}
\Psi^* \eqdef \min_y\left[ \Psi(y) := f^*(\mathbf{A}^\top y) - b^\top y\right]
\end{equation}
The primal and dual optimal solutions $\hat x$ and $\hat y$ are connected through
\begin{align}
    \mathbf{A} \hat x = b, \quad
     \hat x = \nabla f^*(\mathbf{A}^\top \hat y)
    \label{eq:opt_sol}
\end{align}
and the optimal solution set $\mathcal{Y}^*$ of~\eqref{eq:DP} is nonempty. Further, the dual function $\Psi$ is unconstrained, differentiable and its gradient is given by the following expression
\begin{equation*}
    \nabla \Psi(y) = \mathbf{A}\nabla f^*(\mathbf{A}^\top y) - b.
\end{equation*}
Moreover, the gradient $\nabla \Psi$ of the dual function is Lipschitz and componentwise Lipschitz continuous w.r.t. the Euclidean norm $\|\cdot\|_2$, with constant 
$L_{\Psi} = \|A\|^2_2$ and $L_{\Psi,i} = \|a_i\|^2_2$, respectively. 
The coordinate descent update applied to $\Psi$ from \eqref{eq:DP} reads \cite{necoara2016,nesterov2012efficiency}:
\begin{align*}
    y_{k+1} &= y_k - \dfrac{1}{L_{\Psi,i}} e_i \nabla_i \Psi(y_k) \\
    &= y_k - \dfrac{\langle a_i, \nabla f^*(\mathbf{A}^\top y_k) \rangle - b_i}{\|a_i\|^2_2} \cdot e_i
\end{align*}

By using the following relations: 
\begin{align*}
    x^*_k = \mathbf{A}^\top y_k, \quad 
    x_k =  \nabla f^*(x^*_k)
\end{align*}
we just showed that dual coordinate descent iterates transfer to Bregman-Kaczmarz iterates as shown in Algorithm~\ref{alg:BK}.
\begin{algorithm}[ht]
  \caption{Bregman-Kaczmarz method (BK)}
  \label{alg:BK}
  \begin{algorithmic}[1]
    \STATE {choose  $x_0 \in \RR^n$ and set $x^*_0 = x_0$.} 
    \STATE{\textbf{Output:} (approximate) solution of $\text{min}_{\mathbf{A}x=b} \,\, f(x)$ }
    \STATE initialize $k=0$
    \REPEAT
    \STATE choose a row index $i_k=i\in [m]$ (cyclically or randomly)
    \STATE update $x_{k+1}^* = x_{k}^* - \tfrac{\langle a_i,x_k \rangle-b_i}{\|a_i\|^2_2} \cdot a_i$ 
    \STATE update $x_{k+1} = \nabla f^*(x^*_{k+1})$
    \STATE increment $k =  k+1$
    \UNTIL{a stopping criterion is satisfied}
    \RETURN $x_{k+1}$
  \end{algorithmic}
\end{algorithm}

A more general version of randomized coordinate descent, namely APPROX have been proposed in~\cite{fercoq2015accelerated} to accelerate proximal coordinate descent methods for the minimization of composite functions and we state it as Algorithm~\ref{alg:ACD}. We concentrated on that to build our accelerated schemes in the primal space by transferring the dual iterates to the primal space by the relation $c_{k} = A^{\top}v_{k}$ and this results in Algorithm~\ref{alg:ARBK}.

\begin{algorithm}[ht]
  \caption{Dual Accelerated coordinate descent method (ACD)}
  \label{alg:ACD}
  \begin{algorithmic}[1]
    \STATE{\textbf{Input:} Choose  $y_0$ and set $z_0 = y_0$ and  $\theta_0 = \frac{1}{m}$, $b\in \RR^m$, $\mathbf{A} \in \RR^{m\times n}$.} 
    \STATE{ \textbf{Output:} (approximate) solution of  $\min_y \Psi(y)$ }
    \STATE initialize $k=0$
    \REPEAT
        \STATE update $v_k = (1 - \theta_k)y_k + \theta_k z_k    $
		\STATE choose a row index $i_k=i\in [m]$ with probability $p_i=\norm[2]{a_{i}}^{2}/\norm[F]{\mathbf{A}}^{2}$
    \STATE update $z_{k+1} = z_{k} - \frac{a_{i_k}^\top \nabla f^*(A^{\top}v_k) - b_{i_k}}{m\theta_k \|a_{i_k}\|^2_2} e_{i_k}$
    \STATE update $y_{k+1} = v_{k} + m \theta_k (z_{k+1} - z_k)$
    \STATE update $\theta_{k+1} = \dfrac{\sqrt{\theta_k^4 + 4\theta_k^2} - \theta_k^2}{2}$
    \STATE increment $k =  k+1$
    \UNTIL{a stopping criterion is satisfied}
    \RETURN $y_{k+1}$
  \end{algorithmic}
\end{algorithm}

\begin{algorithm}[ht]
\vspace{0.1cm}
  \caption{ Accelerated Randomized Bregman Kaczmarz method (ARBK)}
  \label{alg:ARBK}
  \begin{algorithmic}[1]
    \STATE{\textbf{Input:} Choose  $x_0^*$ and set $t_0 = x_0^*$ and  $\theta_0 = \frac{1}{m}$, $b\in \RR^m$, $\mathbf{A} \in \RR^{m\times n}$.} 
    \STATE{\textbf{Output:} (approximate) solution of $\min_x f(x) \quad \text{s.t.} \quad  \mathbf{A}x=b$ }
    \STATE initialize $k=0$
    \REPEAT
        \STATE update $c_k = (1 - \theta_k)x_k^* + \theta_k t_k    $
		\STATE choose a row index $i_k=i\in [m]$ with probability $p_i=\norm[2]{a_{i}}^{2}/\norm[F]{\mathbf{A}}^{2}$
    \STATE update $t_{k+1} = t_{k} - \tfrac{1}{m\theta_k \|a_{i_k}\|^2_2} (a_{i_k}^\top \nabla f^*(c_k) - b_{i_k})a_{i_k}$
    \STATE update $x_{k+1}^* = c_{k} + m \theta_k (t_{k+1} - t_k)$
    \STATE update $\theta_{k+1} = \dfrac{\sqrt{\theta_k^4 + 4\theta_k^2} - \theta_k^2}{2}$
    \STATE increment $k =  k+1$
    \UNTIL{a stopping criterion is satisfied}
    \RETURN $x_{k+1} = \nabla f^*(x_{k+1}^*)$
  \end{algorithmic}
\end{algorithm}

\begin{remark}
We have written the algorithms in a unified framework to emphasize their similarities. Practical implementations consider only two variables: $(c_k, t_k)$ for ARBK and $(v_k,z_k)$ for ACD. Despite their similarity, the two methods are used for two different purposes. Algorithm~\ref{alg:ACD} outputs $y_k$, an approximate solution of the dual problem~\eqref{eq:DP}, whereas Algorithm~\ref{alg:ARBK} returns $x_k$, an approximate solution of our primal problem~\eqref{eq:PB}.
\end{remark}

We  also use the following relation of the sequences, that is, the iterates of Algorithms~\ref{alg:ARBK} and~\ref{alg:ACD} satisfy for all $k \geq 1$,
\begin{equation}
\label{eq:primal_iterate}
    x^*_{k} = \mathbf{A}^\top y_k, \quad x_k = \nabla f^*(x^*_k)
\end{equation}
Furthermore, Algorithm~\ref{alg:BK} can be recovered from Algorithm~\ref{alg:ARBK} by setting
  $\theta_k = \theta_{0}$ for all $k$. In Algorithm~\ref{alg:BK}, setting $f(x) = \tfrac{1}{2}\|x\|^2_2$ give us the RK method while $f(x) = \lambda\|x\|_1 + \tfrac{1}{2}\|x\|^2_2$ give us the RSK method.

\section{Convergence results for accelerated Randomized Bregman-Kaczmarz method}
\label{sec:convergence}
\noindent In this section we first review basic convergence result of APPROX (ACD), which will be used later to
build convergence result for our method. We first recall the following properties on the sequence $\{\theta_k\}$.

\begin{lemma}~\cite{fercoq2016restarting}
\label{lm:sequence}
The sequence $(\theta_k)$ defined by $\theta_0 \leq 1$ and  $\theta_{k+1} = \frac{\sqrt{\theta_k^4 + 4\theta_k^2} - \theta_k^2}{2}$ satisfies
\begin{equation}
\begin{aligned}
\small
   \frac{(2-\theta_0)}{k + (2-\theta_0)/\theta_0} &\leq \theta_k \leq \frac{2}{k + 2/\theta_0}, \\ 
   \frac{1-\theta_{k+1}}{\theta^2_{k+1}} &= \frac{1}{\theta^2_k}, \quad \forall k=0,1,\dots \\  
   \theta_{k+1} &\leq \theta_k, \quad \forall k=0,1,\dots
\end{aligned}
\end{equation}
\end{lemma}
\medskip

\begin{lemma}~\cite{fercoq2015accelerated}
\label{lm:fercoq}
Let $y_k, z_k$ the sequence generated by ACD, $\theta_{0} = \frac{1}{m}$ and any $\hat y \in \mathcal{Y}^*$. Then it holds 
\begin{align}
\small
\frac{1}{\theta^2_{k-1}}\mathbb{E} [\Psi(y_{k}) - \Psi^*] + \frac{1}{2\theta^2_0} \mathbb{E}[\|z_k - \hat y\|^2_{\mathbf{B}}] \nonumber \\  \leq \frac{1-\theta_0}{\theta^2_0} (\Psi(y_{0}) - \Psi^*) + \frac{1}{2\theta^2_0}\|y_0 - \hat y\|^2_{\mathbf{B}}
\end{align}
with ${\mathbf{B}} = \textbf{Diag}(\|a_1\|^2_2, \|a_2\|^2_2, \dots, \|a_m\|^2_2)$, where $\textbf{Diag}(d_1, d_2,\dots , d_m)$ denote the diagonal matrix with $d_1,d_2,\dots, d_m$ on the diagonal.
\end{lemma}

The following lemma gives the relation between the primal function $f$ and the dual function $\Psi$.
\begin{lemma}
\label{lm:IterateNormSuboptDual}
Let $b\in\mathcal{R}(\mathbf{A})$ and $(x_k^*,x_k)$ be a sequence such that  $x_k = \nabla f^*(x_k^*)$. Then, for any  $y_k\in\mathbb{R}^m$ with $x_k^* = \mathbf{A}^\top y_k$ it holds 
\begin{align}
D_f^{x_k^*}(x_k,\hat x) = \Psi(y_k)- \Psi^*.
\end{align}
\end{lemma}

\begin{proof}
By Definition~\ref{def:D} we have that:
\begin{align*}
D_f^{x_k^*}(x_k,\hat x) &= f^*(x_k^*) + f(\hat x) - \langle x_k^*,\hat x\rangle \\
&= f^*(A^{\top}y_k) - \langle b,y_k\rangle  + f(\hat x) \\
&= \Psi(y_k) + f^*,
\end{align*}
and since  $\Psi^* = -\max -\Psi = -f^*$ (by strong duality), the assertion follows.
\end{proof}

Recall that for our proposed method, we are interested in showing convergence results for iterates $x_k$ given in Equation~\eqref{eq:primal_iterate} and not for iterates $y_k$. By showing that the ACD and ARBK methods are equivalent, and using the recent theoretical results~\cite{necoara2016,nesterov2012efficiency,fercoq2015accelerated} and Lemma~\ref{lm:IterateNormSuboptDual}, we obtain the following convergence results for ARBK as a byproduct.

\begin{theorem}
\label{thm:1}
	Let $x_k$ be the sequence generated by ARBK and let set $\theta_k = \theta_{0},\forall k$. Then, it holds:
\begin{align}
	\small
   \frac{1}{2}\mathbb{E}[\|x_k -\hat{x}\|^2_2] \leq\mathbb{E}[D_f^{x^*_k}(x_k,\hat{x})] \leq \frac{2\norm[F]{\mathbf{A}}^{2}}{k+4}R^{2}_{0}(y_0).
\end{align}
where $\small R_{0}(y_0) = \text{max}_{y} \{\text{min}_{\hat y \in \mathcal{Y}*} \|y - \hat y\|_{2}: \Psi(y) \leq \Psi(y_{0})\}.$
\end{theorem}
\medskip

\begin{proof}
    This rate follows from classical results~\cite{nesterov2012efficiency} for coordinate descent methods and applying Lemma~\ref{lm:IterateNormSuboptDual} and Eq~\eqref{eq:D}.
\end{proof}
\medskip

\begin{theorem}
\label{thm:2}
Let $x_k$ the sequence generated by ARBK, $\theta_0 = \tfrac{1}{m}$ and any $\hat y \in \mathcal{Y}^*$. Then, it holds: 

\begin{align}
\small
\mathbb{E}[D_f^{x^*_k}(x_k,\hat{x})] \leq  \frac{4m^2}{(k-1+2m)^2} C_0,
\end{align}
\begin{align}
\small
\mathbb{E}[\|x_k -\hat{x}\|^2_2] \leq   \frac{8m^2}{(k-1+2m)^2} C_0,
\end{align}
where $$C_0 = \bigg(1 - \frac{1}{m}\bigg) D_f^{x^*_0}(x_0,\hat{x}) + \frac{1}{2}\|y_0 - \hat y\|^2_{\mathbf{B}}$$
\end{theorem}
\medskip

\begin{proof}
    This result follows from Lemma~\ref{lm:fercoq}, Lemma~\ref{lm:IterateNormSuboptDual}, Eq~\eqref{eq:D} and the first inequality in Lemma~\ref{lm:sequence}.
\end{proof}
\medskip

\noindent Theorem~\ref{thm:2} shows that the iterates $x_k$ of Algorithm~\ref{alg:ARBK} converge at the rate $\mathcal{O}(1/k^2)$, thus accelerating its standard counterpart, Algorithm~\ref{alg:BK}, which has $\mathcal{O}(1/k)$ rate of convergence (cf. Theorem~\ref{thm:1}). To the best of our knowledge, accelerated Kaczmarz variants have not yet been
proposed for problem \eqref{eq:PB} and the convergence guarantees for ARBK algorithm from Theorem~\ref{thm:2} are new.


\section{Experiments}
\label{sec:numerics}
\noindent We present several experiments to demonstrate the effectiveness of Algorithm \ref{alg:ARBK} under various conditions. In particular, we study the effects of the sparsity parameter $\lambda$ and the condition number $\kappa$ of the matrix $\mathbf{A}$. The simulations were performed in \texttt{Python} on an Intel Core i7 computer with 16GB RAM. For all the experiments we consider $f(x) = \lambda \cdot \|x\|_1 + \frac{1}{2}\|x\|^2_2$, where $\lambda$ is the sparsity parameter and we compared ARBK, Algorithm~\ref{alg:BK} (which is in this case the randomized sparse Kaczmarz method (RSK)) and the Nesterov acceleration of RSK, (NRSK) \cite[Method ACDM in Section 5]{nesterov2012efficiency}. Synthetic data for the experiments is generated as follows:
all elements of the data matrix $\mathbf{A} \in \RR^{m\times n}$ are chosen independent and identically distributed from the standard normal distribution $\mathcal{N}(0, 1)$. We constructed overdetermined, square, and underdetermined linear systems. To construct sparse solutions $\hat x \in \RR^n$, we generate a random vector $y$ from the standard normal distribution $\mathcal{N}(0, 1)$ and we set $\hat x = S_{\lambda}(\mathbf{A}^\top y)$, which comes from Eq~\eqref{eq:opt_sol} and the corresponding right hand side is $b = \mathbf{A}\hat x \in \RR^m$. For each experiment, we run independent trials each starting with the initial iterate $x_0=0$. We measure performance by plotting the average relative residual error $\| \mathbf{A}x - b\|_2/\|b\|_2$ and the average error $\|x_k - \hat x \|_2/\|\hat x\|_2$ against the number of epochs.

\medskip 

\noindent Figure \ref{fig:example} and Figure \ref{fig:example1} show the result for an overdetermined and consistent system where the value $\lambda = 30$ was used. Note that the usual RSK and NRSK variants perform consistently well. Moreover, we observe experimentally that the ARBK method gives us faster convergence.

\medskip 

\noindent Figure \ref{fig:example2} and Figure \ref{fig:example3} show the results  for a well and ill-conditioned underdetermined and consistent system, respectively,  where the values $\lambda = 30$ were used. All methods take advantage of the fact that the vector $\hat x$ is sparse. Moreover, in Figure \ref{fig:example1}, the RSK and NRSK methods do not reduce the residual as fast as the ARBK method.

\medskip 

\noindent Figure \ref{fig:example4} report the performance of RSK, NRSK, and ARBK on the Mnist dataset available in the Tensorflow framework~\cite{abadi2016tensorflow}. We randomly select a datapoint and consider it as our $\hat x$. We use an underdetermined matrix $A$ and show the relative residuals, the relative errors, and the 4 images which correspond to the original image and the reconstruction using different methods.

\begin{figure}[htb]%
    \centering
    \subfloat[\centering Relative Residual]{{\includegraphics[width=0.57\linewidth]{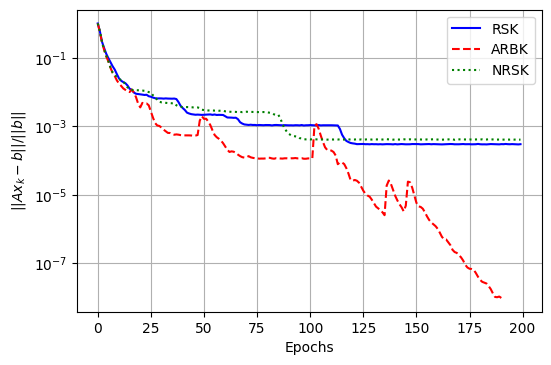} }}%
    \qquad
    \subfloat[\centering Error]{{\includegraphics[width=0.57\linewidth]{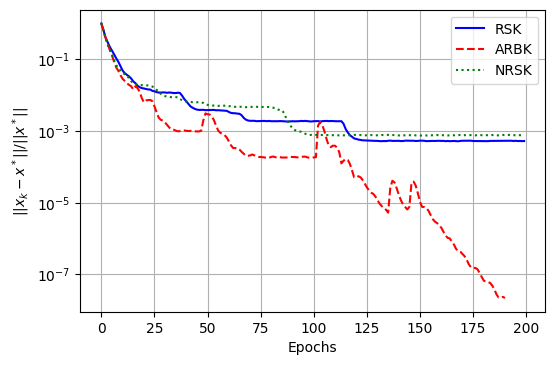} }}%
    \caption{\small A comparison of randomized sparse Kaczmarz (blue), Nesterov acceleration scheme (green) and ARBK method (red), $m = 700, n = 700,$ sparsity $s=182$, $\lambda=30$, $\kappa(A) =1150$.}%
    \label{fig:example}%
    \vspace{-0.4cm}
\end{figure}

\begin{figure}[htb]%
\vspace{0.1cm}
    \centering
    \subfloat[\centering Relative Residual]{{\includegraphics[width=0.57\linewidth]{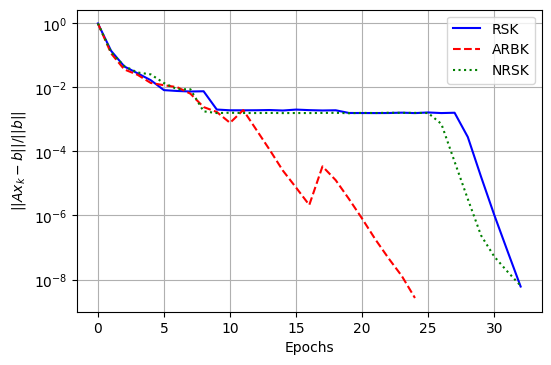} }}%
    \qquad
    \subfloat[\centering Error]{{\includegraphics[width=0.57\linewidth]{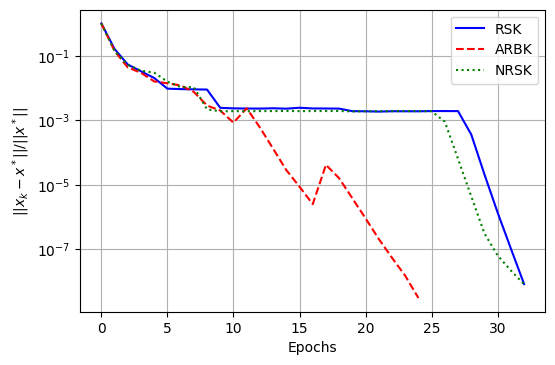} }}%
    \caption{\small A comparison of randomized sparse Kaczmarz (blue), Nesterov acceleration scheme (green) and ARBK method (red), $m = 900, n = 200,$ sparsity $s=65$, $\lambda=30$, $\kappa(A) =2.70$.}%
    \label{fig:example1}%
    \vspace{-0.1cm}
\end{figure}

\begin{figure}[htb]%
    \centering
    \subfloat[\centering Relative Residual]{{\includegraphics[width=0.57\linewidth]{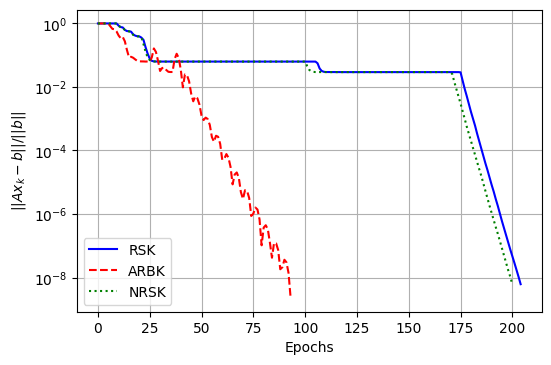} }}%
    \qquad
    \subfloat[\centering Error]{{\includegraphics[width=0.57\linewidth]{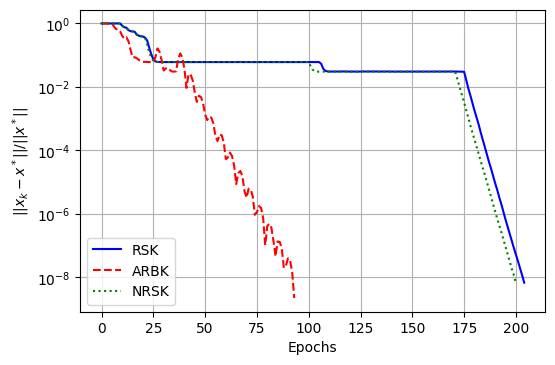} }}%
    \caption{\small A comparison of randomized sparse Kaczmarz (blue), Nesterov acceleration scheme (green) and ARBK method (red), $m = 500, n = 784,$ sparsity $s=7$, $\lambda=60$, $\kappa(A) =8.98$.}%
    \label{fig:example2}%
    \vspace{-0.2cm}
\end{figure}

\begin{figure}[htb]%
\vspace{0.1cm}
    \centering
    \subfloat[\centering Relative Residual]{{\includegraphics[width=0.57\linewidth]{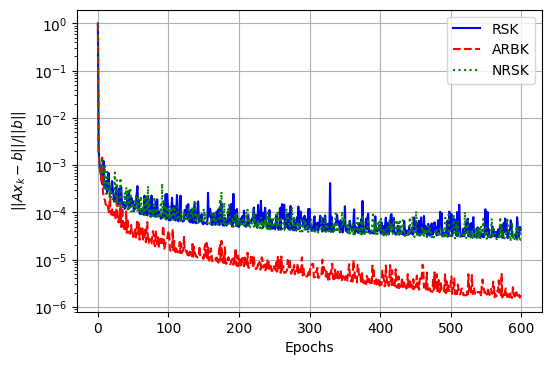} }}%
    \qquad
    \subfloat[\centering Error]{{\includegraphics[width=0.57\linewidth]{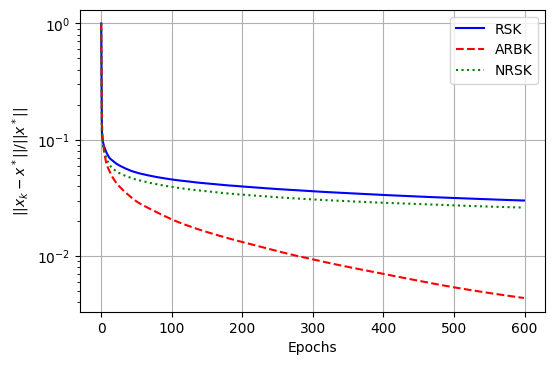} }}%
    \caption{\small A comparison of randomized sparse Kaczmarz (blue), Nesterov acceleration scheme (green) and ARBK method (red), $m = 300, n = 900,$ sparsity $s=231$, $\lambda=15$, $\kappa(A) =2990$.}%
    \label{fig:example3}%
\end{figure}

\begin{figure}[htb]%
\vspace{-0.2cm}
    \centering
    \subfloat[\centering Relative Residual]{{\includegraphics[width=.45\linewidth]{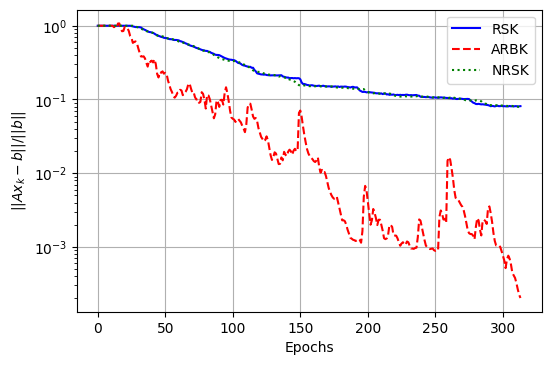} }}
    \subfloat[\centering Error]{{\includegraphics[width=.45\linewidth]{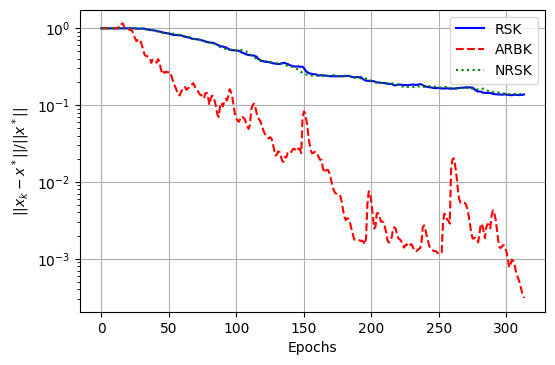} }}\\
    \subfloat[\centering Reconstruction]      {
    \includegraphics[width=.5\linewidth]{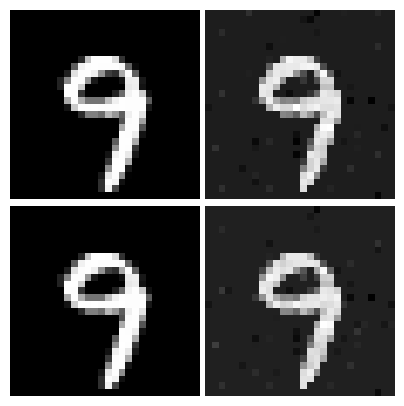}
    }
    \caption{\small A comparison of randomized sparse Kaczmarz(RSK) (blue), Nesterov acceleration scheme(NRSK) (green) and ARBK method (red). On top from left to the right, the original picture, the reconstruction of RSK. On the bottom from left to right the reconstruction of ARBK and the reconstruction of NRSK.$m = 500, n = 784,$ sparsity $s=135$, $\lambda=30$, $\kappa(A) =8.98$.}%
    \label{fig:example4}%
    \vspace{-0.5cm}
\end{figure}

\section{CONCLUSIONS}
\label{sec:conclusion}
\noindent In this paper we have proved that the iterates of the accelerated randomized Bregman-Kaczmarz  method (Algorithm \ref{alg:ARBK}) converge in expectation  at a rate $\mathcal{O}(1/k^2)$ for consistent linear systems.  Numerical experiments show that the method (Algorithm \ref{alg:ARBK}) performs consistently well over a range of  values of $\lambda$, provides very good reconstruction quality as $\lambda$ increases, and demonstrates the benefit of using this method to recover sparse solutions of linear systems.






\section*{ACKNOWLEDGMENT}
\noindent The research leading to these results has received funding from:  ITN-ETN project TraDE-OPT funded by the European Union’s Horizon 2020 Research and Innovation Programme under the Marie Skolodowska-Curie grant agreement No. 861137;  NO Grants 2014-2021, RO-NO-2019-0184, under project ELO-Hyp, contract no. 24/2020;  UEFISCDI PN-III-P4-PCE-2021-0720, under project L2O-MOC, nr. 70/2022.



\bibliographystyle{plain}      
\bibliography{main}   

\end{document}